\documentclass{article}

\usepackage{amsfonts, amsmath, amssymb, amsthm, amscd}
\usepackage{ascmac}

\usepackage{verbatim}

\usepackage{graphics}
\usepackage[all,2cell]{xypic}
\usepackage{mathptmx}

\usepackage{helvet}
\usepackage{courier}        
\usepackage{type1cm} 

\objectmargin+{1mm}
\labelmargin+{0.8mm}
\SelectTips{cm}{12}

\theoremstyle{plain}  
\newtheorem{thm}{Theorem}
\newtheorem{con}[thm]{Conjecture}
\newtheorem{cor}[thm]{Corollary}
\newtheorem{lem}[thm]{Lemma}
\newtheorem{prop}[thm]{Proposition}

\theoremstyle{definition}

\newtheorem{df}[thm]{Definition}

\theoremstyle{remark}

\DeclareMathOperator{\cA}{\mathcal{A}}
\DeclareMathOperator{\cB}{\mathcal{B}}
\DeclareMathOperator{\cC}{\mathcal{C}}
\DeclareMathOperator{\calD}{\mathcal{D}}
\DeclareMathOperator{\cE}{\mathcal{E}}

\DeclareMathOperator{\calH}{\mathcal{H}}
\DeclareMathOperator{\cI}{\mathcal{I}}

\DeclareMathOperator{\calL}{\mathcal{L}}
\DeclareMathOperator{\cM}{\mathcal{M}}

\DeclareMathOperator{\calR}{\mathcal{R}}

\DeclareMathOperator{\cT}{\mathcal{T}}
\DeclareMathOperator{\cU}{\mathcal{U}}

\DeclareMathOperator{\bbC}{\mathbb{C}}

\DeclareMathOperator{\bbE}{\mathbb{E}}
\DeclareMathOperator{\bbF}{\mathbb{F}}

\DeclareMathOperator{\bbH}{\mathbb{H}}

\DeclareMathOperator{\bbK}{\mathbb{K}}
\DeclareMathOperator{\bbL}{\mathbb{L}}

\DeclareMathOperator{\bbQ}{\mathbb{Q}}
\DeclareMathOperator{\bbR}{\mathbb{R}}
\DeclareMathOperator{\bbS}{\mathbb{S}}
\DeclareMathOperator{\bbT}{\mathbb{T}}

\DeclareMathOperator{\bbZ}{\mathbb{Z}}

\UseAllTwocells

\def\sn{\smallskip\noindent}
\def\mn{\medskip\noindent}

\def\enumidef{\renewcommand{\labelenumi}{$\mathrm{(\roman{enumi})}$}}

\newcommand{\Ab}{\operatorname{\bf Ab}}

\newcommand{\BD}{\operatorname{BD}}
\newcommand{\Betti}{\operatorname{Betti}}
\newcommand{\bfone}{\mathbf{1}}

\newcommand{\CH}{\operatorname{CH}}
\newcommand{\Ch}{\operatorname{\bf Ch}}
\newcommand{\ch}{\operatorname{ch}}
\newcommand{\cl}{\operatorname{cl}}

\newcommand{\Cone}{\operatorname{Cone}}

\newcommand{\crys}{\operatorname{crys}}

\newcommand{\dg}{\operatorname{dg}}
\newcommand{\dgCat}{\operatorname{\bf dgCat}}

\newcommand{\EM}{\operatorname{EM}}

\newcommand{\etale}{\operatorname{\text{\'e}t}}
\newcommand{\ev}{\operatorname{ev}}

\newcommand{\Fr}{\operatorname{Fr}}
\newcommand{\Fun}{\operatorname{Fun}}

\newcommand{\Gal}{\operatorname{Gal}}

\newcommand{\Hdg}{\operatorname{Hdg}}

\newcommand{\HK}{\operatorname{HK}}
\newcommand{\HP}{\operatorname{HP}}
\newcommand{\HN}{\operatorname{HN}}

\newcommand{\Ho}{\operatorname{Ho}}

\newcommand{\Hom}{\operatorname{Hom}}
\newcommand{\Homo}{\operatorname{H}}

\newcommand{\id}{\operatorname{id}}

\newcommand{\im}{\operatorname{Im}}
\newcommand{\Ind}{\operatorname{Ind}}
\newcommand{\isoto}{\overset{\scriptstyle{\sim}}{\to}}

\newcommand{\Jac}{\operatorname{Jac}}

\newcommand{\Ker}{\operatorname{Ker}}
\newcommand{\KGL}{\operatorname{KGL}}
\newcommand{\KM}{\operatorname{KM}}

\newcommand{\loc}{\operatorname{loc}}

\newcommand{\Mod}{\operatorname{\bf Mod}}

\newcommand{\Mot}{\mathcal{M}\!ot}

\newcommand{\NMM}{\operatorname{NMM}}
\newcommand{\NM}{\operatorname{NM}}
\newcommand{\Nil}{\operatorname{Nil}}
\newcommand{\nil}{\operatorname{nil}}

\newcommand{\Num}{\operatorname{Num}}
\newcommand{\num}{\operatorname{num}}

\newcommand{\Ob}{\operatorname{Ob}}
\newcommand{\onto}[1]{\stackrel{#1}{\to}}
\newcommand{\op}{\operatorname{op}}

\newcommand{\Perf}{\operatorname{Perf}}

\newcommand{\pr}{\operatorname{pr}}

\newcommand{\ProjSm}{\operatorname{ProjSm}}

\newcommand{\Qcoh}{\operatorname{\bf Qcoh}}

\newcommand{\SH}{\mathcal{SH}}
\newcommand{\Sh}{\mathcal{S}h}

\newcommand{\Spec}{\operatorname{Spec}}
\newcommand{\Sp}{\operatorname{\bf Sp}}
\newcommand{\Spt}{\operatorname{\bf Spt}}

\newcommand{\topo}{\operatorname{top}}

\newcommand{\tr}{\operatorname{Tr}}

\title{Cycle maps on cohomology theories for dg-categories and their applications}
\date{}

\author{Satoshi Mochizuki}

\begin{document}

\maketitle

\begin{abstract}
In this article, we propose noncommutative versions of 
Tate conjecture and Hodge conjecture. 
If we consider these conjectures for a dg-category of perfect complexes 
over a certain schemes $X$, then they are equivalent to 
the classical Tate and Hodge conjectures for $X$ respectively. 
We also propose a strategy of how to prove these conjectures 
by utilizing a version of motivic Bass conjecture.  
\end{abstract}

\section*{Introduction}

For a projective and smooth variety $X$ over a field $k$, the 
$i$-th Chow group $\CH^i(X)$ is generated by cycles of codimension $i$ 
on $X$ modulo rational equivalence. 

\sn
$\mathrm{(i)}$ 
If $k=\bbC$ the field of complex numbers, then 
we can define 
the cycle class morphism from the $i$-th Chow group into 
the $2i$-th Betti cohomology group with rational 
coefficients 
$\rho^i(X)_{\bbQ} \colon \CH^i(X) \otimes \bbQ \to 
\Homo^{2i}(X,\bbQ)$. 
Since $X$ is a K\"ahler manifold, there exists a decomposition 
of its cohomology with complex coefficients 
$$\Homo^k(X,\bbC)=\bigoplus_{p+q=k}\Homo^{p,q}(X)$$
where $\Homo^{p,q}(X)$ is the subgroups of cohomology classes which 
are represented by harmonic form of type $(p,q)$. 
We set $\Hdg^i(X):=\Homo^{2i}(X,\bbQ)\cap \Homo^{i,i}(X)(\subset \Homo^{2i}(X,\bbC))$ and we call it the {\it group of Hodge classes of degree $2i$ on $X$}. 
We can show that $\im \rho^i(X)_{\bbQ}\subset \Hdg^i(X)$. 
In \cite{Hod50}, 
Hodge conjectured the following:

\begin{con}[\bf Hodge conjecture]
\label{con:Hodge conjecture}
For each $i$, we have an equality 
$\im \rho^i(X)_{\bbQ} =\Hdg^i(X)$. 
\end{con}

\sn
$\mathrm{(ii)}$ 
If a prime $l$ is invertible in $k$, we can define 
the cycle class morphism from the $i$-th Chow group into 
the $2i$-th $l$-adic {\'e}tale cohomology group with $i$-th Tate twist 
coefficient 
$\rho^i(X)_{\bbQ} \colon \CH^i(X) \otimes \bbQ \to 
\Homo^{2i}_{\etale}(X,\bbQ_l(i))$. 
It can extends to $\bbQ_i$-linear morphism 
$\rho^i(X)_{\bbQ_l} \colon \CH^i(X) \otimes \bbQ_l 
\to \Homo^{2k}_{\etale}(X,\bbQ_l(i))$. 
In \cite{Tat65}, Tate conjectured following:

\begin{con}[\bf Tate conjecture]
\label{con:Tate conjecture}
For any $i$, the image of the cycle class morphism 
$\rho^i{(X\times_k\Spec \bar{k})}_{\bbQ_l}$ in 
the cohomology group $\Homo^{2i}_{\etale}(X\times_k\Spec \bar{k},\bbQ_i(i))$ 
is exactly the union of fixed parts of 
open subgroups of the Galois group $\Gal(\bar{k}/k)$.
\end{con}

\sn
$\mathrm{(iii)}$ 
If $k$ is a finite field $\bbF_q$ of characteristic $p$ with $q=p^m$, 
then there exits a $p$-adic version of Tate conjecture \cite{Mil07}. 
Let $W(k)$ be the associated ring of $p$-typical Witt vectors and we 
set $K:=W(k)[1/p]$ the fraction field of $W(k)$ and 
$H^{\ast}_{\crys}(X):=H^{\ast}_{\crys}(X/W(k))\otimes_{W(k)}K$ the crystalline 
cohomology groups of $X$. 

\begin{con}[\bf $p$-adic Tate conjecture]
\label{con:p-adic Tate conjecture}
The cycle class map $\CH^{\ast}(X)\otimes \bbQ_p\to 
{H^{2\ast}_{\crys}(X)(\ast)}^{\Fr_p}$ 
with values in the $\bbQ_p$-vector subspace of 
those elements which are fixed by the 
crystalline Frobenius $\Fr_p$ is surjective. 
\end{con}

In this article, we propose noncommutative versions of 
Tate conjecture~\ref{con:noncommutative Tate conjecture} 
and Hodge conjecture~\ref{con:noncommutative Hodge conjecture}. 
If we consider these conjectures for a dg-category of perfect complexes 
over a certain schemes $X$, then they are equivalent to 
the classical Tate and Hodge conjectures for $X$ respectively. 
For Tate conjectures, in \cite{Tab18}, 
Tabuada already formulated noncommutative Tate conjecture and 
$p$-adic Tate conjecture for saturated dg-categories over 
a finite field of characteristic $p$. 
His modi operandi are founded upon Thomason's result in \cite{Tho89} and 
calculations of topological cyclic homology respectively 
(see \S~\ref{subsec:Tate conjecture} 
and \S~\ref{subsec:padic Tate conjecture}). 
The key ingredient of our method in this article is the {\it $l$-adic Chern character} established in \cite{TV19}. 
For Hodge conjecture, in \cite{KKP08}, Katzarkov, Kontsevich and Pantev discussed the Hodge theory of 
algebraic noncommutative spaces and Kontsevich stated 
a version of noncommutative Hodge conjecture in \cite{Kon08}. 
Our handiwork in this article is different from it and is based upon 
the {\it topological $K$-theory} and 
the {\it Deligne-Beilinson cohomology theory} for 
dg-categories over $\bbC$ studied in \cite{Bla15}. 
We also propose a strategy of how to prove these conjectures 
by utilizing a version of motivic Bass conjecture.

\sn
\paragraph{Conventions.}
\begin{enumerate}
\enumidef
\item
For a subring $A\subset \bbR$ and non-negative integer $p$, 
we set $A(p):={(2\pi i)}^p A(\subset \bbR)$. 
\item
For an abelian group $\Homo$ and a unital, associative and commutative ring $B$, we write $\Homo_B$ for $\Homo\otimes_{\bbZ}B$. 
For a category $\cC$ enriched over the category of abelian groups, we 
write $\cC_B$ for the category such that  $\Ob\cC_B=\Ob\cC$ and 
$\Hom_{\cC_B}(x,y)={\Hom_{\cC}(x,y)}_B$ for any $x$ and $y\in\Ob\cC_B$. 
\item
For a family of objects $\{E^p\}_{p\in I}$ in an additive category 
indexed by 
a subset $I$ of $\bbZ$ the ring of integers, 
we set $\displaystyle{E^{\ast}:=\bigoplus_{p\in I}E^p}$. 
For a family of objects $\{E^p(q)\}_{(p,q)\in J}$ 
in an additive category 
indexed by a subset $J$ of $\bbZ\times \bbZ$ 
the set of ordered pair of integers, 
we set $\displaystyle{E^{2\ast}(\ast):=\bigoplus_{\substack{p\in \bbZ\\ (2p,p)\in J}}E^{2p}(p)}$.
\item
For a cochain complex $E=E^{\ast}$ in an additive category closed under 
infinite products, 
we adopt the notation $\displaystyle{E^{\ast}[u^{\pm}]:=\prod_{i\in\bbZ}E^{\ast}[-2i]}$.
\end{enumerate}

\section{Images of cycle maps}
\label{sec:cycle maps}

Let $A$ be a commutative associative unital ring and we set 
$S=\Spec A$. 
By a {\it dg-category} ({\it over $A$}), 
we mean a small category enriched over 
the symmetric monoidal category $\Ch(A)$ of 
chain complexes of $A$-modules. 
We denote the category of small $A$-linear dg-categories over $A$ 
and $A$-linear dg-functors by 
$dgCat_A$ and the stable $\infty$-category of spectra by $\Sp$. 
A dg-functor $f\colon\cA\to \cB$ is 
a {\it Morita equivalence} if it induces an equivalence of 
triangulated categories $Lf_!\colon\calD(\cA)\to\calD(\cB)$ 
on derived categories. 
We denote the $\infty$-category of the localization of $dgCat_A$ 
with respect to Morita equivalences by $\dgCat_A$. 

\subsection{Localizing invariants}
\label{subsec:localizing invarinat}

For a stable presentable $\infty$-category $\calD$, 
an $\infty$-functor $E\colon \dgCat_A \to \calD$ is called 
a {\it localizing invariant} if it preserves filtered colimits and 
satisfies localization. We explain the last condition more precisely. 
A sequence of triangulated categories $\cT\onto{i}\cT'\onto{p} \cT''$ is exact 
if $pi\simeq 0$ and if $i$ is fully-faithful and if 
every objects in $\cT''$ is a direct summand of some object in $\cT'/\cT$. 
We say that a sequence $\cA\to \cB\to \cC$ in $dgCat_A$ is called 
{\it exact} if the induced sequence of triangulated categories 
of derived categories of compact objects 
$\calD^c\cA \to\calD^c\cB\to\calD^c\cC$ is exact. 
We say that a $\infty$-functor $E\colon\dgCat_A\to \calD$ satisfies localization if $E$ sends an exact sequence of dg-categories 
$\cA\to\cB\to\cC$ to a cofiber sequence $E(\cA)\to E(\cB)\to E(\cC)$ in 
$\calD$. 
We denote the $\infty$-category of localizing invariant to $\calD$ 
by $\Fun_{\loc}(\dgCat_A,\calD)$. 
In \cite{CT11} and \cite{CT12}, Cisinski and Tabuada constructed 
the stable presentable $\infty$-category $\Mot_A$ and 
a localizing invariant $\cU\colon\dgCat_A\to\Mot_A$ which satisfies the following universality:

\sn
{\it 
For a stable presentable $\infty$-category $\calD$, 
we have induced equivalence of $\infty$-categories 
$${\cU}^{\ast}\colon\Fun^{\operatorname{L}}(\Mot_A,\calD)\simeq \Fun_{\loc}(\dgCat_A,\calD)$$
where the left-hand side denotes the $\infty$-category of colimit-preserving functors.
}

\sn
In the literature \cite[7.5]{CT12}, 
universal property is written by the language of derivators. 
But by utilizing \cite{Coh13}, we can translate universal property of 
$\Mot_A$ by the language of $\infty$-categories as in \cite{BGT13}. 

Typical examples of localizing invariants $\dgCat_A\to\Sp$ are 
$\bbK$, $\HK$, $\HN$ and $\HP$ the {\it non-connective $K$-theory}, 
the {\it homotopy $K$-theory}, the {\it negative cyclic homology theory} and 
the {\it periodic cyclic homology theory} respectively. 
There exists natural maps 
$\bbK\to \HK$, 
$\ch\colon\bbK\to\HN$ and $\HN\to \HP$. 

\sn
$\mathrm{(i)}$ 
We let $l$ be a prime number invertible in $A$. 
We denote the $\infty$-category of constructible $\bbQ_l$-complexes 
on the {\'e}tale site $S_{\etale}$ of $S$ by $\calL_{\etale}(S_{\etale},l)$. 
It is a symmetric monoidal $\infty$-category and 
we denote by $\Sh_{A,\bbQ_l}:=\Ind(\calL_{\etale}(S_{\etale},l))$. 
We let $\bbT:=\bbQ_l[2](1)$ and we consider the 
$\bbE_{\infty}$-ring object 
$\displaystyle{\bbQ_l(\beta):=\bigoplus_{n\in\bbZ}\bbT^{\otimes n} (=\bbQ_l[\beta,\beta^{-1}])}$ in $\Sh_{A,\bbQ_l}$. 
We denote the $\infty$-category of $\bbQ_l(\beta)$-modules by 
$\Mod_{\bbQ_l(\beta)}(\Sh_{A,\bbQ_l})$. 
In \cite{BRTV18}, we construct the 
$\infty$-functors $r_l\colon \dgCat_A\to \Mod_{\bbQ_l(\beta)}(\Sh_{A,\bbQ_l})$ 
which is called the 
{\it $l$-adic realization functor} and 
the {\it geometric realization functor} 
$|-|\colon \Sh_{A,\bbQ_l}\to \Sp$. 
Moreover we define a natural transformation 
$\ch_l\colon \HK(-)\to |r_l(-)|$ from the 
homotopy invariant $K$-theory $\HK$ to $|r_l(-)|$ which we call 
{\it $l$-adic Chern character} (see \cite[2.3.1]{TV19}). 

\sn
$\mathrm{(ii)}$ 
If $A=\bbC$, in \cite{Bla15}, Blanc constructed 
a localizing invariant $\bbK^{\topo}\colon\dgCat_{\bbC}\to\Sp$ the 
{\it topological $K$-theory} and showed that 
the composition $\bbK\onto{\ch}\HN\to\HP$ factors through 
$\bbK\to\bbK^{\topo}\onto{\ch^{\topo}}\HP$. 
We define {\it Deligne-Beilinson cohomology} 
$\Homo_{\BD}\colon\dgCat_A\to \Sp$ by the 
pull-back of the diagram $\bbK^{\topo}\onto{\ch^{\topo}}\HP \leftarrow \HN$ 
in $\Fun_{\loc}(\dgCat_A,\Sp)$ 
(compare \cite[4.33]{Bla15}). 
By definition, there exists the natural maps 
$\ch_{\BD}\colon\bbK\to\Homo_{\BD}$ and 
$\pr_{\bbK^{\topo}}\colon\Homo_{\BD}\to\bbK^{\topo}$.

\subsection{Relation between the classical and noncommutative motive theories}
\label{subsec:relation}

In this subsection, 
we recall relationship between the classical and noncommutative 
motive theories. 
First recall that Kontsevich introduce the triangulated category $\NMM_A$ of 
{\it noncommutative mixed motives 
over $A$} in 
\cite{Kon05}, \cite{Kon06}, 
\cite{Kon09} and \cite{Kon10}. 
We say that a triangulated full subcategory of a triangulated category is 
{\it thick} if it is closed under direct summands. 
This notion is equivalent to the notion of {\it \'epaisse subcategories} 
of triangulated categories in the sense of Verdier \cite{Ver77} by 
Rickard's criterion \cite[1.3]{Ric89}. 
There exists a natural fully faithful 
embedding $\NMM_A$ into $\Mot^{\loc}_A$ whose 
essential image is the thick triangulated subcategory spanned by 
$\cU(\cA)$ for all saturated dg-categories $\cA$ 
(see \cite[8.5]{CT12}). 
Here recall that a dg-category $\cA$ is {\it saturated} if for any two 
objects $x$ and $y$ in $\cA$, the complex $\cA(x,y)$ 
is a perfect complex of $A$-modules and the right dg 
$\cA^{\op}\otimes^{\bbL}\cA$-module 
$\cA(-,-)\colon\cA^{\op}\otimes^{\bbL}\cA\to\Ch_{\dg}(A)$ 
is a perfect dg $\cA^{\op}\otimes^{\bbL}\cA$-module where $\Ch_{\dg}(A)$ is 
the dg-category of chain complexes of $A$-modules. 
We denote the category of projective and smooth varieties over $A$ by 
$\ProjSm_A$. 
For any projective and smooth scheme $X$ over $A$, 
let $\Perf_X$ be the dg-category of perfect complexes of quasi-coherent 
sheaves on $X$ which is fibrant with respect to the model structure on $\Ch(\Qcoh_X)$ the 
category of chain complexes of quasi-coherent sheaves on $X$ whose 
cofibrations are monomorphisms and 
whose weak-equivalences are quasi-isomorphisms. 
Then by the work of To\"en and Vaqui\'e \cite[3.27]{TV07}, 
the dg-category $\Perf_X$ is saturated. 
In particular $\cU(\Perf_X)$ is in $\NMM_R$. 

If a category which we concern admits both structures of a symmetric monoidal category and of a triangulated category, then May proposed suitable compatibility axioms of both structures in \cite[\S 4]{May01}. 
For example, since $\Ho(\Mot_A)$ is a homotopy category of a certain symmetric monoidal model category, as illustrated in \cite[\S 6]{May01} 
we can show that 
the symmetric monoidal triangulated category $\Ho(\Mot_A)$ satisfies 
the compatibility axioms 
$\mathrm{(TC 1)}$, $\mathrm{(TC 2)}$ and $\mathrm{(TC 3)}$ 
in \cite[p.47-49]{May01}. 
Since $\NMM_A$ is a thick subcategory of $\Ho(\Mot_A)$, $\NMM_A$ also 
satisfies these axioms. 

Next we review the notion of dualizable objects in 
symmetric monoidal categories. 
An object $x$ in a symmetric monoidal category 
$\calR$ is called {\it dualizable} 
if there exists an object $x^{\vee}$ in $\calR$ and morphisms 
$\ev\colon x\otimes x^{\vee}\to \bfone$ and 
$\delta\colon \bfone\to x^{\vee}\otimes x$ such that compositions 
$
x\isoto x\otimes\bfone \onto{\id\otimes\delta}x\otimes x^{\vee}\otimes x
\onto{\ev\otimes \id}\bfone\otimes x\isoto x
$ and 
$
x^{\vee}\isoto\bfone\otimes x^{\vee}\onto{\delta\otimes\id}x^{\vee}\otimes x\otimes x^{\vee}\onto{\id\otimes\ev} x^{\vee}\otimes\bfone \isoto x^{\vee}
$ 
are identities. 
The object $x^{\vee}$ is called the {\it dual of $x$}. 
A typical example of dualizing objects is a 
non-commutative motive $\cU(\cA)$ in $\Ho(\Mot_A)$ associated with 
a saturated dg-category $\cA$ over $A$ by \cite[4.8]{CT12}. 
In particular by lemma~\ref{lem:dualizing objects} below, 
all objects in $\NMM_A$ are dualizable. 

\begin{lem}
\label{lem:dualizing objects}
Let $\cT$ be an idempotent complete triangulated category. Then
\begin{enumerate}
\enumidef 
\item 
If $x\oplus y$ is dualizable, then $x$ and $y$ are also dualizable.

\sn
Moreover assume that $\cT$ satisfies May's axioms $\mathrm{(TC1)}$, 
$\mathrm{(TC2)}$ and $\mathrm{(TC3)}$ 
in {\rm{\cite[p.47-49]{May01}}}. Then 

\item
{\rm (\cite[8.3]{Voe10}).}\ 
In a distinguished triangle $x\to y\to z\to \Sigma x$ in $\cT$, 
if two of $x$, $y$ and $z$ are dualizable, then third one is also. 

\item
In particular, 
If $x$ is dualizable, 
then $\Sigma^nx$ is dualizable for any integer $n$. 
\end{enumerate}
\qed
\end{lem}

Thus we obtain the functor $\NM\colon\ProjSm_A\to\NMM_A$ which sends an object 
$X$ to ${\cU(\Perf_{\dg}X)}^{\vee}$. 
For a projective and smooth scheme $X$ over $A$, we call $\NM(X)$ the 
{\it noncommutative motive associated with $X$}. 

Next we define $K$-motive functor. 
Let $\SH_A$ be a $\infty$-category of Morel-Voevodsky motivic stable homotopy 
category over $\Spec A$. 
We denote the homotopy $K$-theory spectrum in $\SH_A$ by $\KGL$ and 
we write $\Mod_{\KGL}(\SH_{A})$ for the $\infty$-category of $\KGL$-modules. 
We define $\KM\colon\ProjSm_A\to \Ho(\Mod_{\KGL}(\SH_A))$ to be a functor 
by sending a scheme $X$ in $\ProjSm_A$ to 
$\Sigma^{\infty}X_{+}\wedge \KGL$ in $\Ho(\Mod_{\KGL}(\SH_A))$ the homotopy category of $KGL$-modules. 
For a projective and smooth scheme $X$ over $A$, we call $\KM(X)$ the {\it $K$-motive associated with $X$}. 

In \cite{Tab14}, 
Tabuada constructed the comparison functor between $\KM$ and $\NM$. 
We denote the smallest triangulated category which contains $\NMM_A$ and 
closed under infinite direct sums in $\Ho(\Mot_A)$ by $\NMM_A^{\oplus}$. 
If $A=k$ is a perfect field, then he construct the 
fully faithful symmetric monoidal triangulated functor 
$\Phi\colon{\Ho(\Mod_{\KGL}(\SH_k))}_{\bbQ}\to\NMM_{A,\bbQ}^{\oplus}$ which makes 
the diagram below commutative
$$
{\footnotesize{\xymatrix{
 & \ProjSm_A \ar[ld]_{KM_{\bbQ}} \ar[rd]^{NM_{\bbQ}}\\
{\Ho(\Mod_{\KGL}(\SH_k))}_{\bbQ} \ar[rr]_{\ \ \ \ \ \Phi} & & \NMM_{A,\bbQ}^{\oplus}.
}}}
$$
As a formal consequence, we obtain the following Proposition. 

\begin{prop}
\label{prop:compatibility of Chern character}
Let $k$ be a perfect field and let 
$E\colon \dgCat_k \to \Sp$ be a localizing invariant and 
let $\ch_E\colon \bbK\to E$ be a map of localizing invariants. 
Assume that there exists a motivic spectrum $L$ in $\SH_k$ such that
\begin{enumerate}
\enumidef
\item
$\displaystyle{L^{2\ast}(\ast)}$ admits a $KGL_{\bbQ}$-module structure, 
\item
Restriction of $E$ on $\NMM_k$ is represented by $L^{2\ast}(\ast)$. 
Namely we have an equality 
$$\displaystyle{E\simeq \Hom_{\NMM_k^{\oplus}}(-,\Phi(L^{2\ast}(\ast))}.$$ 
\item
Restriction of 
$\ch_E$ on $\Ho(\Mod_{\KGL}(\SH_k))$ corresponds to the structure map $\ch_L\colon\KGL_{\bbQ}\to L^{2\ast}(\ast)$.
\end{enumerate}
Then for any projective and smooth variety $X$, the following diagram is commutative
\begin{equation}
\label{eq:compatibility of cycle maps}
{\footnotesize{\xymatrix{
{(\pi_0(\bbK(\Perf_{\dg}X))}_{\bbQ} \ar[r]^{\pi_0(\cl_{E}(\Perf_{\dg}X))} \ar[d]_{\wr} & 
{\pi_0(E(\Perf_{\dg}X))}_{\bbQ} \ar[d]^{\wr}\\
{\CH^{\ast}(X)}_{\bbQ} \ar[r]_{\cl_{L}} & {L^{2\ast}(X,\ast)}_{\bbQ}.
}}}
\end{equation}
\qed
\end{prop}

\subsection{Noncommutative Tate conjecture}
\label{subsec:noncommutative Tate conjecture}

Let $k$ be a countable perfect field and 
let $l$ be a prime number invertible in $k$. 
Then for any projective and smooth variety $X$ 
over $k$, we have isomorphisms 
$\pi_0(|r_l(\Perf_{\dg}X)|)\simeq \Homo_{\etale}^{2\ast}(X,\bbQ_l(\ast))$ 
(see \cite[\S 2.2, \S 2.3]{TV19}) 
and $\HK(\Perf_{\dg}X)\simeq\bbK(X)$ 
(see for example \cite[p.477]{Wei89} or \cite[2.3]{Cis13}). 
On the other hand, 
there exists a motivic spectrum $\cE_{\etale,l}$ in 
$\SH_k$ such that $\cE_{\etale,l}^{p,q}(X)\simeq 
\Homo^p_{\etale}(X\times_k\bar{k},\bbQ_l(q))$ 
for any projective smooth variety over $k$ 
(see \cite[\S 2.1.5, \S 3.3]{CD12}). 
$\cE_{\etale,l}^{2\ast}(\ast)$ admits a unique 
$KGL_{\bbQ}$-module structure 
(see \cite[\S 2.2]{TV19}). 
From these observations and 
Proposition~\ref{prop:compatibility of Chern character}, 
we propose the following conjecture:

\begin{con}[\bf Noncommutative Tate conjecture]
\label{con:noncommutative Tate conjecture}
Let $k$ be a perfect field and let $l$ be a prime number invertible 
in $k$ and let $\cA$ be a saturated dg-category over $k$. 
Then the image of 
$$\ch_l\otimes_k\bar{k}\colon
{\HK_0(\cA\otimes_k\bar{k})}_{\bbQ_l}\to 
\pi_0(|r_l(\cA\otimes_k\bar{k})|)$$
is just the 
$\Gal(\bar{k}/k)$-invariant part of 
$\pi_0(|r_l(\cA\otimes_k\bar{k})|)$. 
\end{con}

Argument above implies the following.

\begin{cor}
\label{cor:Tate equivalence}
Let $X$ be a projective and smooth variety over a countable perfect field $k$. 
Let $l$ be a prime number invertible in $k$. 
Then Tate conjecture for $X$ and $l$ is equivalent to Tate conjecture 
for $\Perf_{\dg}X$ and $l$. 
\qed
\end{cor}

\subsection{Noncommutative Hodge conjecture}
\label{subsec:noncommutative Hodge conjecture}

Let $X$ be a projective and smooth variety over $\bbC$ of dimension $d$. 
We have an isomorphism $\pi_0\Homo_{\BD}(\Perf_{\dg}X))=
\Homo^{2\ast}_{\BD}(X,\bbQ(\ast))$ (see \cite[4.36]{Bla15}). 
On the other hand, there exists a motivic spectrum $\calH_{\BD}$ in 
$\SH_{\bbC}$ such that ${\calH_{\BD}^{p,q}(X)}_{\bbQ}\simeq \Homo^p_{\BD}
(X,\bbQ(q))$ and ${\calH^{2\ast}(\ast)}_{\bbQ}$ admits a $\KGL$-module 
structure (see \cite[3.6]{HS15}).

For an integer $0\leq i\leq d$, we set 
$\Jac^i(X):=H^{2i-1}(X,\bbC)/
(\Homo^{2i-1}(X,\bbZ(i))+F^i\Homo^{2i-1}(X,\bbC))$ 
and call it the {\it $i$-th intermediate Jacobian of $X$}. 
There exists a short exact sequence (see \cite[7.9]{EV88}).
\begin{equation}
\label{eq:JacBDHod}
0\to{\Jac^{\ast}(X)}_{\bbQ}\to\Homo_{\BD}^{2\ast}(X,\bbQ(\ast))\to\Hdg^{\ast}(X)\to 0
\end{equation}

Recall that there exists a natural map $\pr_{\bbK^{\topo}}\colon\Homo_{\BD}\to \bbK^{\topo}$. 

\begin{df}[\bf Noncommutative Jacobian, Noncommutative Hodge class group]
\label{df:noncom Jacobian, noncom Hcg}
For a saturated dg-category $\cA$, 
we set $\Jac(\cA):=\Ker \pi_0(\pr_{\bbK^{\topo}}(\cA))$ and 
$\Hdg(\cA):=\im \pi_0(\pr_{\bbK^{\topo}}(\cA))$ and 
call them the {\it Jacobian of $\cA$} and the {\it Hodge class group of $\cA$} 
respectively.
\end{df}

By definition, there exists a short exact sequence 
\begin{equation}
\label{eq:JacBDHod2}
0\to\Jac(\cA)\to\Homo_{\BD}(\cA)\to\Hdg(\cA)\to 0.
\end{equation}
The terminologies above are justified by the following lemma.

\begin{lem}
\label{lem:comatibility of notations}
Let $X$ be a projective and smooth variety over $\bbC$. 
Then 
there exists canonical isomorphisms 
$\displaystyle{{\Jac(\Perf_{\dg}X)}_{\bbQ}\simeq{\Jac^{\ast}(X)}_{\bbQ}}$ and 
$\displaystyle{{\Hdg(\Perf_{\dg}X)}_{\bbQ}\simeq\Hdg^{\ast}(X)}_{\bbQ}$ 
which makes diagram below commutative:
\begin{equation}
\label{eq:Jac Deligne Hodge sequence}
{\footnotesize{\xymatrix{
0 \ar[r] & {\Jac(\Perf_{\dg}X)}_{\bbQ} \ar[r] \ar[d]_{\wr} & 
{\Homo^0_{\BD}(\Perf_{\dg}X)}_{\bbQ} \ar[r] \ar[d]_{\wr} & 
{\Hdg(\Perf_{\dg}X)}_{\bbQ} \ar[r] \ar[d]^{\wr} & 0\\
0 \ar[r] & {\Jac^{\ast}(X)}_{\bbQ} \ar[r] & 
{\Homo^{2\ast}_{\BD}(X,\bbQ(\ast))} \ar[r] & 
{\Hdg^{\ast}(X)} \ar[r] & 0.
}}}
\end{equation}
\end{lem}

\begin{proof}
We recall that $i$-th Beilinson-Deligne complex of sheaves 
${\bbQ(i)}_{\BD}(X)$ is given by the following formula
$${\bbQ(i)}_{\BD}(X)=\Cone(\underline{{\bbQ(i)}_X}\oplus F^i\Omega_X^{\ast}\to\Omega_X^{\ast}).$$
Thus there exists a commutative diagram of distinguished triangles
$$
{\footnotesize{\xymatrix{
{\bbQ(i)}_{\BD}(X) \ar[r] \ar[d] & \underline{\bbQ_X} \ar[r] \ar[d] &
\Omega_X^{\ast}/F^i\Omega_X^{\ast} \ar[r] \ar[d]^{\wr} &
\Sigma {\bbQ(i)}_{\calD}(X) \ar[d]\\
F^i\Omega_X^{\ast} \ar[r] & \Omega_X^{\ast} \ar[r] & 
\Omega_X^{\ast}/F^i\Omega_X^{\ast} \ar[r] & \Sigma F^i\Omega_X^{\ast}
}}}
$$
in the derived category of $\bbQ$-sheaves on $X$. 
We denote $\HP_{\ast}(X)$, $\HN_{\ast}(X)$ and 
$\Homo_{\Betti}^{-\ast}(X,\bbQ)$ for the complex of $\bbZ$-modules which 
calculates the periodic cyclic homology and the negative $K$-theory and 
rational Betti cohomology complex of $X$ respectively. 
There exists canonical isomorphisms 
$\displaystyle{\HP_{\ast}(X)\simeq \bbH^{-\ast}(X,\Omega_X^{\ast})[u^{\pm}]}$ and 
$\displaystyle{\HN_{\ast}(X)\simeq \prod_{i\leq 0}\bbH^{-\ast}(X,F^i\Omega_X^{\ast})[-2i]}$ (see \cite[2.7, 3.3]{Wei97}). 
We set $\displaystyle{\Homo_{\BD}^{\ast}(X,\bbQ):=\bbH^{\ast}(X,\prod_{i\geq 0}{\bbQ(i)}_{\BD}(X)[2i])}$. 
By taking the hyper-cohomology on $X$, applying the shift $[-2i]$ and 
taking the product on all integer $i$, we obtain the commutative diagram of 
distinguished triangles
$$
{\footnotesize{\xymatrix{
\Homo_{\BD}^{-\ast}(X,\bbQ) \ar[r] \ar[d] & 
\Homo_{\Betti}^{-\ast}(X,\bbQ)[u^{\pm}] \ar[r] \ar[d] & 
\prod_i \bbH^{-\ast}(X,\Omega_X^{\ast}/F^i\Omega_X^{\ast})[-2i] \ar[r] \ar[d]^{\wr} &
\Sigma \Homo_{\BD}^{-\ast}(X,\bbQ) \ar[d]\\
\HN_{\ast}(X) \ar[r] & \HP_{\ast}(X)  \ar[r] & 
\prod_i \bbH^{-\ast}(X,\Omega_X^{\ast}/F^i\Omega_X^{\ast})[-2i] \ar[r] & \Sigma \HN^{\ast}_{\ast}(X)
}}}
$$
in the derived category of $\bbQ$-vector spaces. 
By applying 
$\EM\colon\calD(\Ab)\to\Ho(\Sp)$ the Eilenberg-Maclane functor 
(see \cite[B.1]{SS03}) to the top distinguished triangle above and 
by the isomorphism $\bbK^{\topo}(X)\wedge_{\bbS}\EM \bbQ\simeq \EM(\Homo^{-\ast}_{\Betti}(X,\bbQ)[u^{\pm}])$ (see in the proof of Proposition~4.36 in \cite{Bla15}), 
we obtain the distinguished triangle 
$${\scriptstyle{\Homo_{\BD}(\Perf_{\dg}X)\wedge \EM\bbQ\onto{\pr_{\bbK^{\topo}}\wedge\EM\bbQ} \bbK^{\topo}(\Perf_{\dg}X)\wedge \EM\bbQ \to \EM(\prod_i \bbH^{-\ast}(X,\Omega_X^{\ast}/F^i\Omega_X^{\ast})[-2i]) \to \Sigma \Homo_{\BD}(\Perf_{\dg}X)\wedge \EM\bbQ}}.$$
By taking $\pi_0$, we obtain the canonical isomorphism 
$\displaystyle{{\Hdg(\Perf_{\dg}X)}_{\bbQ}\simeq\Hdg^{\ast}(X)}_{\bbQ}$ 
which makes the diagram below commutative
$$
{\footnotesize{\xymatrix{
{\pi_0\Homo_{\BD}(\Perf_{\dg}X)}_{\bbQ} \ar[r] \ar[d]_{\wr} & 
{\Hdg(\Perf_{\dg}X)}_{\bbQ} \ar[d]^{\wr}\\
\Homo_{\BD}^{2\ast}(X,\bbQ(\ast)) \ar[r] & {\Hdg(X)}_{\bbQ}.
}}}
$$
Thus by short exact sequences $\mathrm{(\ref{eq:JacBDHod})}$ and 
$\mathrm{(\ref{eq:JacBDHod2})}$, we obtain the 
desired isomorphism 
$\displaystyle{{\Jac(\Perf_{\dg}X)}_{\bbQ}\simeq{\Jac^{\ast}(X)}_{\bbQ}}$. 
\end{proof}

From observations above and 
Proposition~\ref{prop:compatibility of Chern character}, 
we propose the following conjecture: 

\begin{con}[\bf Noncommutative Hodge conjecture]
\label{con:noncommutative Hodge conjecture}
For a saturated dg-category $\cA$ over $\bbC$, 
the cycle map ${K_0(\cA)}_{\bbQ}\to {\Hdg(\cA)}_{\bbQ}$ is surjective. 
\end{con}

By argument above, we obtain the following:  

\begin{cor}
\label{cor:Hodge equivalence}
Let $X$ be a projective and smooth variety over $\bbC$. 
Then Hodge conjecture for $X$ is equivalent to Hodge conjecture 
for $\Perf_{\dg}X$. 
\qed
\end{cor}

\section{Strategy of how to prove conjectures}
\label{sec:applications}

In the article \cite{MY20}, we give a proof of the 
following theorem (compare \cite{RS19}). 
It is a version of {\it motivic Bass conjecture} (see \cite{MY20}). 
For a dg-algebra $B$ over $A$, 
We write $\underline{B}$ for the dg-category with 
the one object $\ast$ and such that 
$\underline{B}(\ast,\ast)=B$. 
We denote $\cU(\underline{A})$ in 
$\Ho(\Mot_A)$ the homotopy category of $\Mot_A$ by $\bfone$. 

\begin{thm}
\label{thm:motivic homotopy type of saturated dg-category}
For any saturated dg-category $\cA$ over $A$, 
there exists a non-negative integer $n$ such that 
$\cU(\cA)$ is a direct summand of $\bfone^{\oplus n}$ in 
$\Ho(\Mot_A)$ the homotopy category of $\Mot_A$. 
\qed
\end{thm}

Theorem~\ref{thm:motivic homotopy type of saturated dg-category} 
enable us that questions for saturated dg-categories 
over $A$ 
or proper and smooth schemes over $\Spec A$ reduce to questions 
for nonconnective $K$-theory of the base ring $A$. 
In this section, we will illustrate how to reduce several such questions to 
the questions for the base ring $A$. 

\begin{proof}[Proof of Conjecture~\ref{con:noncommutative Tate conjecture} and Conjecture~\ref{con:noncommutative Hodge conjecture}]
Since for Conjecture~\ref{con:noncommutative Tate conjecture} and 
Conjecture~\ref{con:noncommutative Hodge conjecture}, 
questions are closed under direct sums and direct summands, 
we shall assume $\cA=\underline{k}$ where $k$ is a base field and in this case 
conjectures are well-known. 
\end{proof}

In particular, by Corollary~\ref{cor:Tate equivalence} and Corollary~\ref{cor:Hodge equivalence}, 
we obtain the following corollaries.

\begin{cor}[\bf Hodge conjecture]
\label{cor:Hodge conjecture}
Hodge conjecture is true.
\qed
\end{cor}

\begin{cor}[\bf Tate conjecture]
\label{cor;Tate conjecture}
For a projective and smooth variety $X$ over a countable perfect field $k$ and 
for a prime number $l$ invertible in $k$, Tate conjecture is true.
\qed
\end{cor}

\subsection{Lattice conjecture}
\label{subsec:Lattice conjecture}

The following conjecture is called the {\it lattice conjecture} (see
 \cite[8.6]{Kal11} and \cite[4.25]{Bla15}). 

\begin{cor}[\bf Lattice conjecture]
\label{cor:lattice conjecture}
Let $\cA$ be a saturated dg-category. 
Then the map 
$$\ch^{\topo}\wedge_{\bbS}\EM \bbC\colon \bbK^{\topo}(\cA)\wedge_{\bbS} \EM\bbC \to 
\HP(\cA)$$ 
is an equivalence where $\EM \bbC$ stands for the 
Eilenberg-Maclane spectrum associated to $\bbC$. 
\end{cor}

\begin{proof}
Since $\bbK^{\topo}$ and $\HP$ factor through $\Mot_{\bbC}$ and since 
this question is closed under direct sums and direct summands, 
we shall assume that $\cA=\underline{\bbC}$. 
In this case, the result is known 
(see remark after Conjecture 7.25 in \cite{Bla15}).
\end{proof}

\subsection{Noncommutative Parshin conjecture}
\label{subsec:Parshin conjecture}

In \cite{Gei08}, \cite{Gei15}, Geisser proved that 
Tate conjecture over finite fields implies Parshin conjecture. 
Kahn also gives a similar remark about logical connections between 
Bass conjecture and Parshin conjecture for proper and smooth varieties 
over finite fields and Beilinson-Soul\'e vanishing conjecture \cite{Sou85} 
(see for example \cite[Theorem 39 and p.396]{Kah05}). 
In this subsection, 
by using (a version of) motivic Bass conjecture, 
we will also show 
the following noncommutative Parshin conjecture 
which is proposed by Tabuada in \cite{Tab18}. 
By virtue of Tabuada's result \cite[1.3]{Tab18}, 
it implies the original Parshin conjecture 
(more precise statement see below).

\begin{cor}[\bf Noncommutative Parshin conjecture]
\label{cor:noncommutative Parshin conjecture}
Let $\cA$ be a saturated dg-category over a finite field $\bbF_q$. 
Then $\bbK_n(\cA)_{\bbQ}=0$ for every $n\neq 0$ where $\bbK_n(\cA)$ 
stands for the $n$-th non-connective algebraic $K$-theory of $\cA$. 
\end{cor}

\begin{proof}
Let $\cA$ be a saturated dg-category over a finite field $\bbF_q$. 
Since the question is closed under direct sums and direct summands, 
by Theorem~\ref{thm:motivic homotopy type of saturated dg-category}, 
we shall assume that 
$\cA=\underline{\bbF_q}$. 
Then by calculation of algebraic $K$-theory of finite fields by 
Quillen \cite{Qui72}, we obtain the result.
\end{proof}

In particular we will obtain the following original Parshin conjecture. 

\begin{cor}[\bf Parshin conjecture]
\label{cor:parshin conjecture}
Let $X$ be a proper and smooth variety over a finite field. 
Then $i$-th algebraic $K$-group ${K_i(X)}_{\bbQ}$ of $X$ and $j$-th motivic 
cohomology group $\Homo_{\cM}^{j}(X,\bbQ(k))$ are trivial for all 
positive integer $i>0$ and a pair of integers $j\neq 2k$.
\qed
\end{cor}

Similarly we obtain the following:

\begin{cor}
Assume that $R$ is a regular ring. 
Then for any saturated dg-category $\cA$ over $R$, 
the $-n$-th non-connective $K$-group 
$\bbK_{-n}(\cA)$  of $\cA$ is trivial for any $n>0$. 
\end{cor}

\begin{proof}
As in the proof of non-connective Parshin conjecture, 
we shall assume that $\cA=\underline{R}$. 
Then the result is well-known 
(see for example \cite[Proposition 6.8 (b)]{TT90} or \cite[Example 9.5]{Sch06}).\end{proof}

\subsection{Noncommutative Thomason-Tate conjecture}
\label{subsec:Tate conjecture}

Thomason tried to 
show that Bass conjecture for projective and smooth schemes over 
finite fields implies Tate conjecture for varieties over finite fields 
and as its working hypothesis, 
in the article \cite{Tho89}, he stated the equivalent statement of 
Tate conjecture over finite fields. 
More precisely 
Thomason showed that the following statement equivalent to 
the Tate conjecture on $X$ for all $i$. 
We say that an abelian group $A$ is {\it $l$-reducible} if 
the group $\Hom(\bbZ(l^{\infty}),A)$ is trivial where 
$\bbZ(l^{\infty})$ stands for the the Pr\"ufer $l$-group, 
namely $A$ has no non-trivial $l$-divisible elements. 

\begin{con}[\bf Thomason conjecture]
\label{con:Thomason conjecture}
For all $n$, the group 
$\pi_{-1}L_{KU}K(X\times_{\bbF_q}\Spec \bbF_{q^n})$ 
is $l$-reducible where 
$L_{KU}K(X\times_{\bbF_q}\Spec \bbF_{q^n})$ 
stands for the Bousfield localization of the algebraic $K$-theory spectrum 
of $X\times_{\bbF_q}\Spec \bbF_{q^n}$ 
with respect to topological complex $K$-theory $KU$. 
\end{con}

Here is a noncommutative version of Thomason conjecture which is 
proposed by Tabuada in \cite{Tab18}. 
Tabuada showed that for a proper and smooth variety $X$ over a finite field 
$\bbF_q$, Tate conjecture for $X\times_{\bbF_q}\Spec \overline{\bbF}_q$ 
is equivalent to noncommutative 
Tabuada-Thomason-Tate conjecture below for $\Perf_X$ 
(see \cite[1.3]{Tab18}).

\begin{cor}[\bf Tabuada-Thomason-Tate conjecture]
\label{cor:Tabuada conjecture}
Let $\bbF_q$ be a finite field and $l$ be a prime number 
which is prime to $q$ and let $\cA$ be a saturated dg-category over 
$\bbF_q$. 
Then $\pi_{-1}L_{KU}K(\cA\otimes_{\bbF_q}\bbF_{q^n})$ 
is $l$-reducible for all $n$ 
where $L_{KU}K(\cA\otimes_{\bbF_q}\bbF_{q^n})$ 
stands for the Bousfield localization of 
the algebraic $K$-spectrum if $\cA\otimes_{\bbF_q}\bbF_{q^n}$ with 
respect to topological complex $K$-theory $KU$. 
\end{cor}

\begin{proof}[Proof of Corollary~\ref{cor:Tabuada conjecture}]
Since for each positive integer $n$ and a saturated dg-category $\cA$ over 
$\bbF_q$, 
a dg-category $\cA\otimes_{\bbF_q}\bbF_{q^n}$ is a saturated dg-category 
over $\bbF_{q^n}$, 
by replacing $k=\bbF_q$ with $\bbF_{q^n}$ and $\cA$ with 
$\cA\otimes_{\bbF_q}\bbF_{q^n}$, we shall 
only give a proof for $n=1$. 
Since for any spectra $X$, 
we have an equivalence of spectra (\cite[4.3]{Bou79}) 
$L_{KU}X\simeq X\wedge L_{KU}(\Sigma^{\infty}S^0)$, 
$L_{KU}$ preserves distinguished triangles and direct sums. 
Thus by 
Theorem~\ref{thm:motivic homotopy type of saturated dg-category}, 
we reduced to the case for $\cB=\underline{\bbF_q}$. 
Since Tate conjecture for $\Spec \overline{\bbF}_q$ is true, 
$\pi_{-1}L_{KU}K(\bbF_q)$ is $l$-reducible by 
Thomason's theorem \cite{Tho89}. 
Hence we obtain the result.
\end{proof}

\subsection{Noncommutative $p$-adic Tate conjecture}
\label{subsec:padic Tate conjecture}

Tabuada proposed 
a noncommutative $p$-adic Tate conjecture 
and proved that it implies the original $p$-adic version of 
Tate conjecture in \cite[1.3]{Tab18}. 
Recall the notations from introduction. 
Let $k$ be a finite field $\bbF_q$ of characteristic $p$ with $q=p^m$, 
and 
let $W(k)$ be the associated ring of $p$-typical Witt vectors and we 
set $K:=W(k)[1/p]$ the fraction field of $W(k)$. 
Let $\cA$ be a saturated dg-category over $k$. 
Then Tabuada constructed the $p$-adic noncommutative cycle map 
\begin{equation}
\label{eq:padic noncommutative cycle map}
\rho_{\cA}\colon K_0(\cA)\otimes K\to TP_0(\cA)^{{\varphi}^m}_{1/p}
\end{equation}
where $TP_0(\cA)_{1/p}^{\varphi^m}$ stands for 
$K$-linear subspace of ${TP_0(\cA)}_{1/p}=TP_0(\cA)\otimes_{W(k)}K$ 
the $0$-th topological periodic cyclic homology of $\cA$ consisting of 
those elements which are fixed by the $K$-linear endomorphism 
$\varphi^m$ the $m$-times cyclotomic Frobenius 
(for more detail, see below). 
Here is a noncommutative $p$-adic Tate conjecture and 
we will show it by utilizing motivic Bass conjecture. 
In particular we obtain the $p$-adic Tate conjecture. 

\begin{cor}[\bf Noncommutative $p$-adic Tate conjecture]
\label{cor:noncommutative padic Tate conjecture}
For a saturated dg-category $\cA$ over $k$, the cycle map 
$\mathrm{(\ref{eq:padic noncommutative cycle map})}$ is surjective. 
\end{cor}

Before giving our proof of 
Corollary~\ref{cor:noncommutative padic Tate conjecture}, 
we briefly recall some notations and fundamental results of 
topological periodic cyclic homology from 
\cite{BM17}, \cite{HM97}, \cite{Tab18} and so on. 
We write $THH$ for the topological Hochschild homology theory. 
Since $THH$ is a localizing invariant on the category of 
small dg-categories (over $\bbZ$) by \cite[1.2, 7.1]{BM12}, 
we can regard it as a functor on $\Mot^{\loc}_k$ to $\Ho(\Spt)$ 
the category of spectra by \cite[7.2]{CT12}. 
Topological periodical cyclic homology $TP$ is defined as the 
Tate cohomology of the circle group action on 
topological Hochschild homology. Namely 
$TP(\cA)={THH(\cA)}^{tS^1}$ for any dg-category $\cA$ over $k$. 
We can also regard $TP$ as the symmetric monoidal functor 
from $\Mot_k^{\loc}$ to $\Ho(\Spt)$ by 
\cite[7.5]{CT12} and \cite[Theorem A]{BM17} 
(see also \cite{AMN18}). 
By the work of Hesselholt and Madsen \cite[\S 5]{HM97}, 
we obtain the following computation results. 
In particular for a saturated dg-category over $k$, 
${TP_0(\cA)}_{1/p}:=TP_0(\cA)\otimes_{W(k)}K$ has 
a structure of $K$-vector space. 
\begin{lem}
\label{lem:HM97}
As a graded $W(k)$-algebra, we have an isomorphism 
\begin{equation}
\label{eq:HM97}
TP_{\ast}(k)\isoto S_{W(k)}\{v^{\pm 1}\}
\end{equation}
for some $v\in TP_{-2}(k)$. 
Here $S_{W(k)}$ stands for the symmetric algebra over $W(k)$. 
\qed
\end{lem}

We obtain the cycle class map 
$\rho\colon {K_0(-)}_K(=K_0(-)\otimes K) \to {TP_0(-)}_{1/p}$ by universality of $K$-theory \cite[8.7]{CT12} 
(see also around Lemma 3.7 in \cite{Tab18}). 
We write $\rho_X$ for $\rho_{\Perf_X}$ for a scheme $X$ over $k$ and 
we denote $\rho_{\Spec A}$ by $\rho_A$ for a $k$-algebra $A$. 

In \cite{Tab18}, by utilizing \cite[II 4.2]{NS18}, Tabuada 
constructed the {\it cyclotomic Frobenius morphism} 
$$\varphi\colon{TP_0(\cA)}_{1/p}\to {TP_0(\cA)}_{1/p}$$
for a saturated dg-category $\cA$ over $k$ 
which satisfies the following conditions 
(see around Lemma 3.7 in \cite{Tab18}).

\begin{lem}
\label{lem:Tabuada Frobenius}
Let $\cA$ be a saturated dg-category over $k$. Then
\begin{enumerate}
\enumidef
\item
$\varphi^m$ is a $K$-linear map where $m=\log_p\# k$ 
where $\# k$ stands for the number of elements in $k$. 

\item
The image of the cycle map 
$\rho_{\cA}\colon {K_0(\cA)}_K\to {TP_0(\cA)}_{1/p}$ is in 
${TP_0(\cA)}^{{\varphi}^m}_{1/p}$ the $K$-linear subspace 
of ${TP_0(\cA)}_{1/p}$ of those elements 
which are fixed by the $K$-linear endomorphism ${\varphi}^m$. 

\item
$\varphi^{m}\colon {TP_0(k)}_{1/p}\to {TP_0(k)}_{1/p}$ 
is the identity morphism. 

\item
$\rho_{k}\colon K_0(k)_K\to {TP_0(k)}_{1/p}^{\varphi^m}$ 
is just the identity map of $K$ via 
the canonical isomorphisms ${K_0(k)}_K\simeq K$ and 
${TP_0(k)}_{1/p}^{\varphi^m}\simeq K$.
\end{enumerate}
\qed
\end{lem}

\begin{proof}[Proof of Corollary~\ref{cor:noncommutative padic Tate conjecture}]The question is closed under direct summand, thus by 
Theorem~\ref{thm:motivic homotopy type of saturated dg-category}, 
we reduce to the problem for $\cB=\underline{k}$. 
For this case, assertion is true by 
Lemma~\ref{lem:Tabuada Frobenius} $\mathrm{(iv)}$. 
\end{proof}

\subsection{Noncommutative standard conjectures}
\label{subsec:standard conjecture}

In this subsection, 
We consider the 
noncommutative Beilinson and Voevodsky nilpotence 
conjectures which were proposed by Tabuada  
respectively Bernardara, Marcolli and Tabuada 
in \cite{Tab18} and \cite{BMT18} respectively. 
Let $\cA$ be a saturated dg-category over a field $k$. 
Then there exists a pairing 
$\chi\colon K_0(\cA)\times K_0(\cA)\to\bbZ$ 
which sends a pair $([M],[N])$ to an integer 
$\displaystyle{\sum_i{(-1)}^i\dim_k\Hom_{\calD_c(\cA)}(M,n[-i])}$ where 
$\calD_c(\cA)$ stands for the full subcategory of all compact objects in 
$\calD(\cA)$ the derived category of $\cA$. 
We say that an element $x$ in $K_0(\cA)$ is 
{\it numerically equivalent to zero} if for any element $y$ in $K_0(\cA)$, 
$\chi(x,y)=0$. 
We denote the quotient group of ${K_0(\cA)}_{\bbQ}$ by subgroup of 
all elements which are numerically equivalent to zero by ${K_0(\cA)}_{\num}$. 
Notice that we have the canonical isomorphism $K_0(\cA)\simeq\Hom_{\NMM_k}(\bfone,\cU(\cA))$ and therefore we can regard an element of $K_0(\cA)$ as 
a morphism $\bfone \to\cU(\cA)$. 
An element $x$ in $K_0(\cA)$ is {\it $\otimes$-nilpotent} if there exists an 
integer $n>0$ such that $x^{\otimes n}=0$ where $\otimes$ is the symmetric monoidal structure on $\NMM_k$. We write ${K_0(\cA)}_{\nil}$ 
for the quotient group of $K_0(\cA)_{\bbQ}$ 
by the subgroup of all $\otimes$-nilpotent elements. 
Then we will show the followings.

\begin{cor}[\bf Noncommutative Beilinson conjecture]
\label{cor:noncommutative Beilinson conjecture}
If $k$ is a finite field, then the equality 
${K_0(\cA)}_{\num}={K_0(\cA)}_{\bbQ}$ holds. 
\end{cor}

\begin{cor}[\bf Noncommutative Voevodsky nilpotence conjecture]
\label{cor:noncommutative Voevodsky conjecture}
The equality ${K_0(\cA)}_{\nil}={K_0(\cA)}_{\num}$ 
holds. 
\end{cor}

As in \cite[2.9]{Tat94}, 
the Beilinson conjecture and 
the ($p$-adic) Tate conjecture 
implies the following strong Tate conjecture. 

\begin{cor}[\bf Strong Tate conjecture]
\label{cor:strong Tate conjecture}
Let $X$ be a projective and smooth scheme $X$ over a finite field $k$. 
Then the order of the pole of the Hasse-Weil zeta function $\zeta(x,s)$ of $X$ 
at $-i$ is equal to the dimension of $\CH^i(X)_{\bbQ,\num}$ for $0\leq i\leq \dim X$. 
\qed
\end{cor}

By the work of Bernardara, Marcolli and Tabuada \cite[1.1]{BMT18}, 
noncommutative nilpotence conjecture 
implies the original nilpotence conjecture and as in \cite[p.194]{Voe95}, 
nilpotence conjecture implies 
the standard conjecture D and 
it might be well-known that the standard conjecture D 
implies the standard conjecture B 
(see for example \cite[5.1]{Kle94}). 
Thus finally we obtain the following:

\begin{cor}[\bf Standard conjecture]
\label{cor:Standard conjecture}
Grothendieck's standard conjectures B and D are true. 
\qed
\end{cor}

Before showing corollaries~\ref{cor:noncommutative Beilinson conjecture} 
and \ref{cor:noncommutative Voevodsky conjecture}, 
we will recall a general theory of ideals of ringoids 
which we will use to prove corollaries. 
Let $\calR$ be a ringoid, in other words, 
a category enriched over the 
category of abelian groups. 
An {\it ideal} $\cI$ of $\calR$ is a family of subgroups $\cI(x,y)$ of 
$\Hom_{\calR}(x,y)$ for all ordered pair 
$(x,y)\in\Ob\left(\calR\times\calR\right)$ which 
satisfies the condition that 
for any objects $x$, $y$, $z$ and $u$ of $\calR$ and 
any pair of morphisms $f\colon x\to y$ and $g\colon z\to u$, 
we have an inclusion $g\cI(y,z)f\subset\cI(x,y)$. 
In this situation, we define $\calR/\cI$ to be a ringoid 
by setting $\Ob\calR/\cI:=\Ob\calR$ and 
$\Hom_{\calR/\cI}(x,y):=\Hom_{\calR}(x,y)/\cI(x,y)$ 
for any pair of objects $x$ and $y$ in $\calR$. 
Then the compositions in $\calR$ induces a composition of $\calR/\cI$. 
We call the category $\calR/\cI$ 
the {\it quotient category of $\calR$ by $\cI$}. 

From now on, we assume $\calR$ admits a unital 
symmetric monoidal structure $(\otimes,\bfone)$. 
An {\it $\otimes$-ideal} $\cI$ of $\calR$ is an ideal of $\cI$ such that 
for all triple of objects $x$, $y$ and $z$ and a morphism 
$f\colon x\to y$ in $\cI(x,y)$, 
$\id_z\otimes f$ and $f\otimes\id_z$ are also in 
$\cI(z\otimes x,z\otimes y)$ and $\cI(x\otimes z,y\otimes z)$ respectively. 
We will illustrate typical two examples of $\otimes$-ideals. 
The first one is $\Nil_{\calR}$ which is defined as follows. 
A morphism $f\colon x\to y$ in $\calR$ is {\it $\otimes$-nilpotent} if 
there exits an integer $n>0$ such that $f^{\otimes n}=0$. 
For a pair of objects $x$ and $y$ in $\calR$, 
we denote the set of all $\otimes$-nilpotent morphisms from $x$ to $y$ 
by $\Nil_{\calR}(x,y)$. 
Then we can show that the family 
$\{\Nil_{\calR}(x,y)\}_{(x,y)\Ob(\calR\times\calR)}$ is a $\otimes$-ideal.

Assume that an object $x$ in $\calR$ is dualizable. 
Then for a morphism $f\colon x\to x$, we define $\tr(f)$ 
to be a morphism $\bfone \to \bfone$ by 
compositions $\bfone\onto{\delta}x^{\vee}\otimes x\simeq x\otimes x^{\vee} 
\onto{f\otimes \id_{x^{\vee}}}x\otimes x^{\vee} \onto{ev}\bfone$ and 
we call it the ({\it categorical}) {\it trace of $f$}. 
We say that 
a morphism $f\colon x\to y$ is {\it numerically equivalent to zero} if 
for any morphism $g\colon y\to x$, $\tr(gf)=0$ and we denote the 
set of all morphisms from $x$ to $y$ which are numerically equivalent to zero 
by $\Num_{\calR}(x,y)$. 
Moreover if we assume that all objects in $\calR$ are dualizable, then 
the family $\{\Num_{\calR}(x,y)\}_{(x,y)\in\Ob(\calR\times\calR)}$ 
forms a $\otimes$-ideal by \cite[7.1.1]{AKO02}. 
If $\Hom_{\calR}(\bfone,\bfone)$ is a field, then we have an inclusion 
$\Nil_{\calR}(x,y)\subset \Num_{\calR}(x,y)$ 
for any pair of objects $x$ and $y$ in $\calR$ 
by \cite[7.1.4]{AKO02}. 
We write $\NMM_{A,\nil}$ and $\NMM_{A,\num}$ for 
the quotient categories ${(\NMM_A)}_{\bbQ}/\Nil_{{(\NMM_A)}_{\bbQ}}$ and 
${(\NMM_A)}_{\bbQ}/\Num_{{(\NMM_A)}_{\bbQ}}$ respectively. 
These notions are compatible with the notions in introduction. 
Namely for a saturated dg-categories $\cA$ over a field $k$, 
we have the following isomorphisms 
\begin{equation}
\label{eq:K_0(A)sharp}
K_0(\cA)_{\#}\simeq\Hom_{\NMM_{k,\#}}(\bfone,\cU(\cA))
\end{equation}
for $\#\in\{\nil,\num\}$ by \cite[4.4]{MT14}. 
Moreover $\NMM_{k,\#}$ for $\#\in\{\nil,\num \}$ are symmetric monoidal triangulated categories by Lemma~\ref{lem:quotient category} below.  

\begin{lem}
\label{lem:quotient category}
\begin{enumerate}
\enumidef
\item
Assume that $\calR$ is additive. Then 
for any pair of families of objects $\{x_i\}_{i\in I}$ and 
$\{y_j\}_{j\in J}$ indexed by 
non-empty finite sets $I$ and $J$, 
the canonical isomorphism 
$$\Hom_{\calR}(\bigoplus_{i\in I}x_i,\bigoplus_{j\in J}y_j)\isoto 
\bigoplus_{(i,j)\in I\times J}\Hom_{\calR}(x_i,y_j)$$
induces an isomorphism 
\begin{equation}
\label{eq:ideal and finite sum}
\cI(\bigoplus_{i\in I}x_i,\bigoplus_{j\in J}y_j)\isoto 
\bigoplus_{(i,j)\in I\times J}\cI(x_i,y_j).
\end{equation}
In particular $\calR/\cI$ is an additive category.

\item
Moreover assume that $\calR$ is a triangulated category. 
Then we can make $\calR/\cI$ into a triangulated category where 
a triangle in $\calR/\cI$ is distinguished if and only if it 
is isomorphic to a distinguished triangle in $\calR$. 

\item
{\rm(\cite[6.1.2]{AKO02}).}\ 
If $\calR$ is a symmetric monoidal category and $\cI$ is a $\otimes$-ideal, 
then the symmetric monoidal structure on $\calR$ makes $\calR/\cI$ into 
a symmetric monoidal category. 
\end{enumerate}
\qed
\end{lem}

\begin{proof}[Proof of Corollaries~\ref{cor:noncommutative Beilinson conjecture} and \ref{cor:noncommutative Voevodsky conjecture}]
Since Corollaries~\ref{cor:noncommutative Beilinson conjecture} and 
\ref{cor:noncommutative Voevodsky conjecture} 
are equivalent to assertions 
that 
$$\Num(\bfone,\cU(\cA))_{\bbQ}=0 \text{  and  }
\Nil(\bfone,\cU(\cA))_{\bbQ}=\Num(\bfone,\cU(\cA))_{\bbQ}$$
respectively, if assertions are true for $\cA$, then 
assertions are also true for direct summands of $\cU(\cA)$ 
by the equality $\mathrm{(\ref{eq:ideal and finite sum})}$. 
Thus by Theorem~\ref{thm:motivic homotopy type of saturated dg-category}, 
we shall assume that 
$\cA=\underline{k}$ a finite dg-cell over $k$. 
Since $K_0(k)=\bbZ$, assertions for $\cA=\underline{k}$ are true. 
\end{proof}

\paragraph{Acknowledgement}
The author wish to express my deep gratitude to 
Seidai Yasuda 
for giving several comments and 
he greatly appreciate Kanetomo Sato, Kento Yamamoto, Takashi Suzuki, 
Masana Harada and Kazuya Kato for stimulative 
discussions in the early stage of the works.

\mn
SATOSHI MOCHIZUKI\\
\emph{DEPARTMENT OF MATHEMATICS,
CHUO UNIVERSITY,
BUNKYO-KU, TOKYO, JAPAN.}\\
e-mail: {\tt{mochi@gug.math.chuo-u.ac.jp}}\\

\end{document}